\documentclass[12pt]{amsart}
\usepackage{amsmath,amssymb}
\usepackage{eufrak}
\usepackage{amsthm}
\usepackage{enumitem}
\usepackage{graphicx}
\usepackage{graphics}
\usepackage{color}
\usepackage[dvipsnames]{xcolor}

\usepackage{faktor}
\usepackage{hyperref}

\usepackage[all]{xy}
\xyoption{matrix}
\xyoption{arrow}

\usepackage{youngtab}
\usepackage{marginnote}
\usepackage{lscape}
\usepackage{tikz}
\usetikzlibrary{calc,shadings,patterns}
\usetikzlibrary{matrix}
\usetikzlibrary{arrows}
\usetikzlibrary{matrix}
\usepackage{verbatim}
\usepackage{multirow}
\usetikzlibrary{decorations.pathreplacing}
\usepackage{enumitem,pifont,xcolor}
\definecolor{myblue}{RGB}{0,29,119}

\newtheorem{theorem}{Theorem}[subsection]
\newtheorem{proposition}[theorem]{Proposition}
\newtheorem{corollary}[theorem]{Corollary}
\newtheorem{lemma}[theorem]{Lemma}
\theoremstyle{definition}
\newtheorem{definition}[theorem]{Definition}
\newtheorem{example}[theorem]{Example}

\newtheorem{remark}[theorem]{Remark}

\newtheorem{notation}[theorem]{Notation}

\newtheorem{innercustomthm}{Theorem}
\newenvironment{customthm}[1]
  {\renewcommand\theinnercustomthm{#1}\innercustomthm}
  {\endinnercustomthm}

\newcommand\commentout[1]{{}}

\makeatletter
\@addtoreset{equation}{section}
\makeatother

\newcommand{\Zz}{{\mathbb Z}} 

\newcommand{\Commentout}[1]{} 

\DeclareMathOperator{\findim}{fin.\! dim}

\DeclareMathOperator{\gldim}{gl\,\! dim}

\DeclareMathOperator{\pdim}{pdim}

\newcommand{\cP}{{\mathcal P}}

\newcommand{\cS}{{\mathcal S}}

\providecommand{\AMS}{$\mathcal{A}$\kern-.1667em%
\lower.25em\hbox{$\mathcal{M}$}\kern-.125em$\mathcal{S}$}

\begin{document}

\title[The {\large$\varphi$}-Dimension of cyclic Nakayama algebras]{The {\Large$\varphi$}-Dimension of cyclic Nakayama algebras}
 
\author{Emre SEN}

\address{Department of Mathematics, Northeastern University, Boston MA}
\email{sen.e@husky.neu.edu}

\maketitle
\begin{abstract}
K. Igusa and G. Todorov introduced the $\varphi$ function which generalizes the notion of projective dimension. We study the behavior of the $\varphi$ function for cyclic Nakayama algebras of infinite global dimension. We prove that the supremum of values of $\varphi$ is always an even number. In particular we show that the $\varphi$-dimension is $2$ if and only if the algebra satisfies certain symmetry conditions. Also we give a sharp upper bound for $\varphi$-dimension in terms of the number of monomial relations which describes the algebra.
\end{abstract}
\tableofcontents
\section{Introduction}

The original motivation for introducing the $\varphi$ function in \cite{todo} was to prove  the finitistic dimension conjecture, which states that $sup\{\pdim M | M \text{ in mod-}\Lambda, \text{  }  \pdim M < \infty\}$ 
is finite.
 Using the function $\varphi$ it was proved that 
 for certain classes of artin algebras, in particular for algebras with representation dimension 3 the finitistic dimension conjecture holds. At that time it was  shown that many classes of algebras had representation dimension 3 and hence the conjecture was true for those classes. However, it was shown in \cite{ro} that there are algebras of arbitrary representation dimension, hence the conjecture is still open.
  

It follows from Definition \ref{varphi} and Remark \ref{phi} that the $\varphi$ function is a generalization of projective dimension, in the sense that $\varphi(M)=\pdim(M)$ if $\pdim(M)<\infty$. However, $\varphi(M)$ is a finite integer even when projective dimension is infinite. Using this as a motivation, M. Lanzilotta suggested to treat this function  $\varphi$ as a new invariant of modules of infinite projective dimension, and globally as an invariant of algebras (especially algebras of infinite global dimension). It was proved in \cite{lanz} that selfinjective algebras can be characterized as algebras with $\varphi (M)=0$ for all modules $M$, or $\varphi\dim(\Lambda)=0$ where  
$$\varphi\dim(\Lambda):=sup\{\varphi(M)\ |\ M \in \text{mod-}\Lambda\}.$$ 
Clearly, with this definition, for algebras of finite global dimension it follows that $\varphi\dim\Lambda= \gldim\Lambda$. 

In this paper we concentrate  on cyclic Nakayama artin algebras and study their $\varphi$-dimension.
When global dimension is infinite,  we analyse $\varphi$-dimension and prove the following theorems (see Theorem \ref{even} and Theorem \ref{sharpbound}).

\begin{customthm}{(A)}
Let $\Lambda$ be a cyclic Nakayama algebra of infinite global dimension. Then $\varphi\dim\Lambda$ is always an even number. 
\end{customthm}
\begin{customthm}{(B)}
Let $\Lambda$ be a cyclic Nakayama algebra of infinite global dimension. 
The sharp bound for $\varphi\dim\Lambda$ is given by $2r$ where $r$ is the number of monomial relations defining $\Lambda$.
\end{customthm}

Indeed, we can extend the Theorem (B) to any cyclic Nakayama algebra by using Gustafson's wellknown result \cite{gust} about upper bounds of Nakayama algebras of finite global dimension.

In section 2
we describe cyclic orderings on various classes of modules induced by the cyclic ordering of the vertices of the quiver. 
Throughout this work the crucial role was played by certain ${\bf\Delta}$-modules, see definition \ref{definitionofdelta}, ${\bf\Delta}$-filtrations, ${\bf\Delta}$-socles and many other standard notions modified to the ${\bf\Delta}$ set up.
In section 3, the $\varphi$-dimension is defined in general and also a particularly nice description is given for Nakayama algebras. The proof of Theorem (A) is in section 4 and the proof of Theorem (B) is in section 5. We are thankful to valuable comments and suggestions of anonymous referee. 
\section{Set up and Notation}
Let $\Lambda$ be a cyclic Nakayama algebra over $N\geq 3$ vertices given by $r\geq 2$ many relations $\boldsymbol\alpha_{k_{2i}}\ldots\boldsymbol\alpha_{k_{2i-1}}=0$ where $1\leq i\leq  r$ and $k_{f}\in\left\{1,2,\ldots,N\right\}$ for quiver $Q$ as in the figure 1. Notice that each arrow $\boldsymbol\alpha_i$ starts at the vertex $i$ and ends at the vertex $i+1$ with the exception $\boldsymbol\alpha_N:N\mapsto 1$.

\begin{figure}\centering
\begin{tikzpicture}
\foreach \ang\lab\anch in {90/1/north, 45/2/{north east}, 0/3/east, 270/i/south, 180/{N-1}/west, 135/N/{north west}}{
  \draw[fill=black] ($(0,0)+(\ang:3)$) circle (.08);
  \node[anchor=\anch] at ($(0,0)+(\ang:2.8)$) {$\lab$};
}

\foreach \ang\lab in {90/1,45/2,180/{N-1},135/N}{
  \draw[->,shorten <=7pt, shorten >=7pt] ($(0,0)+(\ang:3)$) arc (\ang:\ang-45:3);
  \node at ($(0,0)+(\ang-22.5:3.5)$) {$\boldsymbol\alpha_{\lab}$};
}

\draw[->,shorten <=7pt] ($(0,0)+(0:3)$) arc (360:325:3);
\draw[->,shorten >=7pt] ($(0,0)+(305:3)$) arc (305:270:3);
\draw[->,shorten <=7pt] ($(0,0)+(270:3)$) arc (270:235:3);
\draw[->,shorten >=7pt] ($(0,0)+(215:3)$) arc (215:180:3);
\node at ($(0,0)+(0-20:3.5)$) {$\boldsymbol\alpha_3$};
\node at ($(0,0)+(315-25:3.5)$) {$\boldsymbol\alpha_{i-1}$};
\node at ($(0,0)+(270-20:3.5)$) {$\boldsymbol\alpha_i$};
\node at ($(0,0)+(225-25:3.5)$) {$\boldsymbol\alpha_{N-2}$};

\foreach \ang in {310,315,320,220,225,230}{
 \draw[fill=black] ($(0,0)+(\ang:3)$) circle (.02);
}
\end{tikzpicture}
\caption{Quiver Q}
\end{figure}
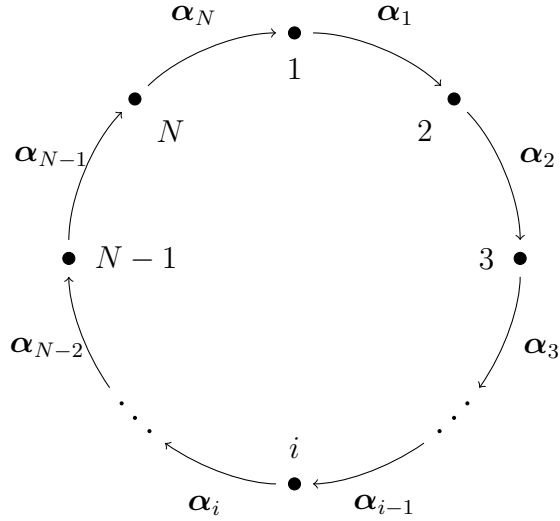

\subsection{Properties of systems of relations}

We assume that the algebra $\Lambda$ is given as the path algebra of the above quiver $Q$ modulo the system of relations REL:
\begin{gather*}
\boldsymbol\alpha_{k_2}\ldots\boldsymbol\alpha_{k_1+1}\boldsymbol\alpha_{k_1}\ \ =0 \\
\boldsymbol\alpha_{k_4}\ldots\boldsymbol\alpha_{k_3+1}\boldsymbol\alpha_{k_3}\ \ =0 \\
\dots \\
\boldsymbol\alpha_{k_{2r-2}}\ldots\boldsymbol\alpha_{k_{2r-3}+1}\boldsymbol\alpha_{k_{2r-3}}=0\\
\boldsymbol\alpha_{k_{2r}}\ldots\boldsymbol\alpha_{k_{2r-1}+1}\boldsymbol\alpha_{k_{2r-1}}=0
\end{gather*}
Hence $\Lambda\cong\faktor{kQ}{\left<REL\right>}$.
First, we assume that this system of relations is irredundant i.e. none of the relations is a consequence of the other relations. 
Second, we order them according to their starting index: $1\leq k_1 < k_3<\ldots<k_{2r-1}\leq N$. This gives 
an ordering on the last indices of relation, which is shown in Proposition \ref{cyclicorder} also to be cyclic.
The induced ordering on the classes of projectives is  called \emph{cyclic ordering of classes of projectives} and is given in Proposition \ref{c-ordering}. 

\begin{remark} Let REL be the above irredundant system of equations. Then:
\begin{enumerate}[label=\arabic*)]
\item There exists at most one relation starting at each $k_j\in[1,N]$.
\item There exists at most one relation ending at each $k_j\in[1,N]$.
\item There is no restriction on the lengths of relations except that the system of relations has to be irredundant. 
\item Because $\Lambda$ is cyclic Nakayama algebra, none of the simple modules are projective hence each relation is composition of at least two arrows. 
\end{enumerate}
\end{remark}

\begin{example}\label{example}
To illustrate the cyclic ordering, consider the relations:
\begin{align}
\boldsymbol\alpha_{3}\boldsymbol\alpha_{2}\boldsymbol\alpha_{1}=0\\
\boldsymbol\alpha_{1}\boldsymbol\alpha_{5}\boldsymbol\alpha_{4}\boldsymbol\alpha_{3}=0
\end{align}
where $N=5$. This gives the indecomposable projectives:
\begin{align}
P_1\cong\begin{vmatrix}
   S_1 \\
   S_2 \\
   S_3
\end{vmatrix}\hookrightarrow
P_5\cong\begin{vmatrix}
  S_5\\
   S_1 \\
   S_2 \\
   S_3
\end{vmatrix}\hookrightarrow
P_4\cong\begin{vmatrix}
S_4\\
S_5\\
   S_1 \\
   S_2 \\
   S_3
\end{vmatrix}, \hspace{20pt}
P_3\cong\begin{vmatrix}
   S_3 \\
   S_4 \\
   S_5\\
   S_1
\end{vmatrix}\hookrightarrow
P_2\cong\begin{vmatrix}
   S_2 \\
   S_3 \\
   S_4\\
   S_5\\
   S_1
\end{vmatrix}
\end{align}
Notice that socles are $S_3$ and $S_1$. As a natural number $1$ is smaller than $3$, but here under cyclic ordering $1>3>1$.
\end{example}


\subsection{Projective and Injective modules}\label{projectives}
Using the above system of relations, we describe indecomposable projective and injective modules. Since the system is irredundant, projective-injective modules will occur as projective covers of the simples labeled by the index next to the first index of each relation. Explicitly they are:
\begin{align}
P_{k_1+1}=I_{k_4},\quad P_{k_3+1}=I_{k_6},\quad \dots \quad P_{k_{2r-3}+1}=I_{k_{2r}},\qquad P_{k_{2r-1}+1}=I_{k_{2}}
\end{align}
Furthermore, we get classes of projective modules characterized by their socles:
\begin{align*}P_{k_{1}}\hookrightarrow\ldots\hookrightarrow  P_{(k_{2r-1})+1}=I_{k_{2}} \quad\text{ have simple } S_{k_{2}} \text{  as their socle}\\
P_{k_3}\hookrightarrow\ldots\hookrightarrow P_{k_1+1}=I_{k_4}  \quad\text{ have simple } S_{k_4} \text{  as their socle}\\
\dots\qquad\qquad\qquad\qquad\qquad\qquad\qquad\\
P_{k_{2r-1}}\hookrightarrow\ldots\hookrightarrow P_{(k_{2r-3})+1}=I_{k_{2r}}\quad  \text{ have simple } S_{k_{2r}} \text{  as their socle}
\end{align*}
Similarly, we get classes of injective modules characterized by their tops:
\begin{align*}
 P_{(k_{2r-1})+1}=I_{k_{2}}\twoheadrightarrow\ldots\twoheadrightarrow I_{k_4+1}  \quad\text{ have simple } S_{k_{2r-1}+1}\text{  as their top}\\
P_{k_1+1}=I_{k_4}\twoheadrightarrow\ldots\twoheadrightarrow I_{k_2+1}   \quad\text{ have simple } S_{k_1+1} \text{  as their top}\\
\dots\qquad\qquad\qquad\qquad\quad\quad\quad\quad\quad\quad\\
P_{(k_{2r-3})+1}=I_{k_{2r}} \twoheadrightarrow\ldots\twoheadrightarrow I_{k_4+1}  \quad\quad\text{ have simple } S_{k_{2r-3}+1} \text{  as their top}
\end{align*}

\begin{notation}\label{classes} For each simple module $S_{k_{2j}}$, the class of projective modules which have  socle $S_{k_{2j}}$, is given as follows. Also we will use the  notation: \\
$cl(j): = (P_{(k_{2j-1})}\hookrightarrow P_{(k_{2j-1})-1}\hookrightarrow\ldots\hookrightarrow P_{(k_{2j-3})+1}=I_{k_{2j}})$. We will denote by:\\
$P_{min}^{(j)}:=P_{(k_{2j-1})}$ the 
module of the shortest length in this class 
and 
by \\
$P_{max}^{(j)}:=P_{(k_{2j-3})+1}=I_{k_{2j}}$
the module of the longest length in this class.
\end{notation}

\begin{remark} Let $r$ be the number of relations in the irredundant system REL. Then:
$r=\#\{cl(j) \ |\ \text{classes of projectives}\}$\\
$r=\#\{S_{k_{2j}} \ |\ \text{simples which are socles of projectives}\}$\\
$r=\#\{P_{min}^{(j)} \ |\ \text{minimal projectives}\}$\\
$r=\#\{P_{max}^{(j)} \ |\ \text{maximal projectives}\}$
\end{remark}

It is easy to see that the projective classes  correspond to the left sides and injective classes to the right sides of the hills occurring at Auslander-Reiten quiver of $\Lambda$.

\subsection{Cyclic ordering} The cyclic ordering of the vertices of the quiver $Q$ induces the cyclic ordering on: simple modules, indecomposable projectives, socles of projectives, minimal projectives and other modules. Below we describe these orderings precisely. Many of the arguments will use these cyclic orderings. In particular Proposition \ref{cyclicorder} is used in an essential way.

\begin{notation} In order to match the visual order on the Auslander-Reiten quiver, we will usually write in decreasing order $\dots > 1> N> (N-1)>\dots > 2> 1> N>\dots$. 
Also sometimes when it is important, we will use notation $a\  \cdot\!\!\!\!>b$  to emphasize that there is no $c$ such that $a>c>b$. 
Auslander-Reiten quiver will not be used here, but this information might be useful to the people who use AR-quiver.
\end{notation}

\begin{definition}  \label{orderingd} Using the above cyclic ordering 
of the vertices, we define cyclic ordering of the simple and projective modules:
$$\dots > S_1> S_N> S_{N-1}>\dots > S_2> S_1> S_N>\dots\quad\quad (*_S)$$
$$\dots > P_1> P_N> P_{N-1}>\dots > P_2> P_1> P_N>\dots\quad\quad (*_P)$$
\end{definition}

\begin{lemma} \label{connecting}
Let $\tau$ be the Auslander-Reiten translation. Then:
\begin{enumerate}[label=\roman*)]
\item  $\tau^{-1}S_i\cong S_{i-1}$ for $i\in\{2,.. ,N\}$
and $\tau^{-1}S_1\cong S_{N}$. 
\item $\tau^{-1}(top P_{max}^{(j)})\cong top P_{min}^{(j-1)}$ for $j\in\{2,..,r\}$
and $\tau^{-1}(top P_{max}^{(1)})\cong top P_{min}^{(r)}$.
\end{enumerate}
\end{lemma}
\begin{proof}(a) This follows by the standard construction of AR-quiver.\\
(b) We have $\tau^{-1}(top P_{max}^{(j)})\cong \tau^{-1}S_{(k_{2j-3})+1} \cong S_{(k_{2j-3})}\cong top P_{min}^{(j-1)}$ by Definition \ref{classes} of $P_{max}^{(j)}$ and $P_{min}^{(j-1)}$ and part (a).
\end{proof}

Using this lemma, the cyclic ordering 
$(*_P)$ of the projectives can be described more precisely in terms of the classes $\{cl(j)\}$ of projectives having the same socle $S_{k_{2j}}$ and shortest and longest projectives $\{P_{min}^{(j)}\}$ and $\{P_{max}^{(j)}\}$.

\begin{proposition} \label{ordering}(a) The inclusions of the projectives within a class $cl(j)$
define  the  ordering of the projectives which agrees with the ordering in $(*_P)$
$$P_{min}^{(j)}=P_{k_{2j-1}}\ \cdot\!\!\!\!>P_{(k_{2j-1})-1}      \  \cdot\!\!\!\!>\ldots\  \cdot\!\!\!\!>P_{(k_{2j-3})+1}=P_{max}^{(j)}.$$
(b) Furthermore $P_{max}^{(j)}\ \cdot\!\!\!\!>P_{min}^{(j-1)}$ for each $j\in\{2,..,r\}$ and $P_{max}^{(1)}\ \cdot\!\!\!\!>P_{min}^{(r)}$.\\
(c) The cyclic ordering of projectives $(*_P)$ can also be given as:
$$\ \cdot\!\!\!\!>P_{min}^{(j)}\ \cdot\!\!\!\!>\!..\ \cdot\!\!\!\!>\!P_{max}^{(j)}\!\ \cdot\!\!\!\!>\!P_{min}^{(j-1)}\!\ \cdot\!\!\!\!>\!..\!\ \cdot\!\!\!\!>\!P_{max}^{(j-1)}\!\ \cdot\!\!\!\!>\!\dots\!\ \cdot\!\!\!\!>\!P_{min}^{(1)}\!\ \cdot\!\!\!\!>\!..\ \cdot\!\!\!\!>\!P_{max}^{(1)}\!\ \cdot\!\!\!\!>\!P_{min}^{(r)}\!\ \cdot\!\!\!\!>\!..\!\ \cdot\!\!\!\!>\!P_{max}^{(r)}\!\ \cdot\!\!\!\!>$$ 
\end{proposition}
\begin{proof} Both parts (a) and (b) follow from the fact that the ordering of the projectives in $(*_P)$ is given by the ordering of the simples. Part (b) also uses 
Lemma \ref{connecting}.
\end{proof}

\subsection{Cyclic ordering on subsets}
From the set of relations REL we have the following subsets of the integers $\{1,\dots, N\}$:\\
$k_{odd}:=\{k_{2r-1}, k_{2r-3}, \dots, k_{2j-1}, \dots, k_3, k_1\}$$=\{$indices of the beginnings of relations$\}$=\\
$\{$labels of the minimal projectives $\{P_{min}^{(j)}\}$ and their simple tops$\}$.\\
$k_{even}:=\{k_{2r}, k_{2(r-1)}, \dots, k_{2j}, \dots, k_4, k_2\}$$=\{$indices of the ends of relations$\}$=\\
$\{$labels of the socles of all projectives$\}$.

We now consider the induced orderings on these subsets and on the associated subsets of projective and simple modules.

\begin{lemma} \label{c-ordering} The cyclic ordering 
$\small{.. > 1> N> (N-1)>.. > 2> 1> N>..}$
of the vertices of the quiver induces the following cyclic orderings:
\begin{enumerate} [label=\roman*)]
\item The set $k_{odd}$ is cyclicaly ordered (by using induced subset ordering) as: \\
$\dots \ \cdot\!\!\!\!>k_{2r-1}\ \cdot\!\!\!\!>k_{2r-3}\ \cdot\!\!\!\!>\dots\ \cdot\!\!\!\!>k_{2j-1}\ \cdot\!\!\!\!>\dots\ \cdot\!\!\!\!>k_3\ \cdot\!\!\!\!>k_1\ \cdot\!\!\!\!>k_{2r-1} \ \cdot\!\!\!\!>\dots$.
\item The cyclic ordering of all projectives induces the following cyclic ordering on the set of minimal projectives $\{P_{min}^{(j)}\}_{j=1}^{r}$:\\
$\dots \ \cdot\!\!\!\!>P_{min}^{(r)}\ \cdot\!\!\!\!>P_{min}^{(r-1)}\ \cdot\!\!\!\!>\dots\ \cdot\!\!\!\!>P_{min}^{(j)}\ \cdot\!\!\!\!>P_{min}^{(j-1)}\ \cdot\!\!\!\!>\dots\ \cdot\!\!\!\!>P_{min}^{(2)}\ \cdot\!\!\!\!>P_{min}^{(1)}\ \cdot\!\!\!\!>P_{min}^{(r)}\ \cdot\!\!\!\!>\dots$.
\end{enumerate}
\end {lemma}

\begin{proof} i) Follows from the fact that the integers in $k_{odd}$ were already chosen with such an ordering. The only new inequality is $k_1\ \cdot\!\!\!\!>k_{2r-1}$. This follows since $k_1>k_{2r-1}$ is a consequence of the minimal transitive property between $k_1$ and $k_{2r-1}$ in the original cyclic ordering $\small{.. > 1> N> (N-1)>.. > 2> 1> N>..}$.

ii) In order to see $P_{min}^{(j)}\ \cdot\!\!\!\!>P_{min}^{(j-1)}$ consider the ordering in Proposition \ref{ordering} (a) and (b) and notice that there are no other $P_{min}^{(t)}$ between $P_{min}^{(j)}$ and $P_{min}^{(j-1)}$.
\end{proof}

\begin{proposition} \label{cyclicorder} Let $\cS$ be a complete set of representatives of isomorphism classes of simple modules indexed by the set $k_{even}$ i.e. $\mathcal S\!:=\!\{S_{k_2}, S_{k_4},..,S_{k_{2j}},..,S_{k_{2r}}\}$. Then:
\begin{enumerate}[label=\roman*)]
\item The following cyclic ordering of the simple modules in $\mathcal S$ is induced by the cyclic ordering of all simple modules:
$$\dots \ \cdot\!\!\!\!>S_{k_{2r}}\ \cdot\!\!\!\!>S_{k_{2(r-1)}}\ \cdot\!\!\!\!>\dots\ \cdot\!\!\!\!>S_{k_{2j}}\ \cdot\!\!\!\!>S_{k_{2(j-1)}}\ \cdot\!\!\!\!>\dots \ \cdot\!\!\!\!> S_{k_{2}}\ \cdot\!\!\!\!>S_{k_{2r}}\ \cdot\!\!\!\!>\dots$$
\item The induced cyclic ordering of the end terms of the relations in REL 
is:
$$\dots\ \cdot\!\!\!\!>k_{2r}\ \cdot\!\!\!\!>k_{2(r-1)}\ \cdot\!\!\!\!>\dots\ \cdot\!\!\!\!>k_{2j}\ \cdot\!\!\!\!>k_{2(j-1)}\ \cdot\!\!\!\!>\dots \ \cdot\!\!\!\!>k_{2}\ \cdot\!\!\!\!>k_{2r}\ \cdot\!\!\!\!>\dots.$$
\end{enumerate}
\end {proposition}
\begin{proof} i) To see that $S_{k_{2j}}\ \cdot\!\!\!\!>S_{k_{2(j-1)}}$ we need the following claim.\\
Claim: Suppose $S_{a}\ \cdot\!\!\!\!>S_{b}$ for two simple modules in $\mathcal S$. Then $P^{(a)}_{min}\ \cdot\!\!\!\!>P^{(b)}_{min}$.\\
Proof of the claim: 
The entire section of irreducible maps from $S_a$ to the maximal projective $P^{(a)}_{max}$ 
is to the left in the Auslander-Reiten quiver from the section of irreducible maps from $S_b$ to the maximal projective  $P^{(b)}_{max}$ 
and there is no such section which is between these two sections. Therefore there are no projectives between these sections, and hence there are no projectives between $P^{(a)}_{max}$ and $P^{(b)}_{min}$ and consequently no minimal projectives between $P^{(a)}_{min}$ and $P^{(b)}_{min}$. Therefore $P^{(a)}_{min}\ \cdot\!\!\!\!>P^{(b)}_{min}$.

Suppose $S_{k_{2j}}\ \cdot\!\!\!\!>S_{k_{2(j-1)}}$  is not true. Then there is another simple $S_t\in \mathcal S$ where $S_t\not\cong S_{k_{2j}}, S_{k_{2(j-1)}}$ so that 
$S_{k_{2j}}\ \cdot\!\!\!\!>S_t >S_{k_{2(j-1)}}$. 
This implies $P_{min}^{(j)} \ \cdot\!\!\!\!>P_{min}^{(t)}$ by the claim which contradicts $P_{min}^{(j)} \ \cdot\!\!\!\!>P_{min}^{(j-1)}$ as was shown in Lemma \ref{c-ordering}.

\noindent ii) This follows from i) since the ordering of simple modules is given by their indices.
\end{proof}

\subsection{Notion of $\bf\Delta$-modules
} Of particular importance for the study of $\varphi$-dimension are the $\bf\Delta$-modules which are obtained as extensions of a particular collection of modules ${\bf\Delta}=\{\Delta_1,..,\Delta_r\}$, which we now describe.

Recall that the simple modules $\{S_{k_2}, S_{k_4},\dots, S_{k_{2r}}\}$ are the socles of indecomposable projective modules as in Notation \ref{classes}. Since these modules, and also certain closely related modules, will play a very important role, we recall and introduce the following notation.

\begin{notation}
Let  $\Lambda$ be a cyclic Nakayama algebra given by the set of $r$ relations. We define the following set of representatives of isomorphism classes of  simple modules:
$$ \mathcal S'\!:=\!\{S_{k_2+1}, S_{k_4+1},..,S_{k_{2j}+1},..,S_{k_{2r}+1}\}.$$
Recall that, previously we used $\mathcal S=\!\{S_{k_2}, S_{k_4},..,S_{k_{2j}},..,S_{k_{2r}}\}$.
\end{notation}

\noindent In the Example \ref{example} we have: $\cS=\{S_1,S_3\}$, $\cS'=\{S_2,S_4\}$.
\begin{definition} \label{definitionofdelta}let $\Lambda$ be  a cyclic Nakayama algebra defined by the system of $r$ relations REL.
For each $j\in\{1,\dots,r\}$ let $\Delta_j$ be a shortest indecomposable uniserial module with 
$soc\Delta_j\cong S_{k_{2j}}$ and $top\Delta_j\cong S_{k_{2(j-1)}+1}$. Let ${\bf\Delta}=\{\Delta_1,\Delta_2,..,\Delta_j,..,\Delta_r\}$, 
i.e. a set of shortest modules of the following form:


${\bf\Delta}:=\left\{ \Delta_1\cong\begin{vmatrix}
    S_{k_{2r}+1} \\
    \vdots  \\
    S_{k_{2}}
\end{vmatrix}\!, \Delta_2\cong\begin{vmatrix}
    S_{k_{2}+1}  \\
    \vdots  \\
   S_{k_{4}}
\end{vmatrix}\!,..,\Delta_j\cong\begin{vmatrix}
   S_{k_{2(j-1)}+1}  \\
    \vdots  \\
    S_{k_{2j}}
\end{vmatrix}\!,..,\Delta_r\cong \begin{vmatrix}
   S_{k_{2r-2}+1}  \\
    \vdots  \\
    S_{k_{2r}}
\end{vmatrix}\!\right\}.$\\
We will use term $\bf\Delta$-module for any module which is isomorphic to a module which has a filtration by the modules in the set $\bf\Delta$.
\end{definition}

\begin{lemma}\label{simplesindelta} The simple composition factors of distinct $\Delta_j$s are disjoint.
\end{lemma}
\begin{proof}
By definition, each $\Delta_j$  is the shortest module with given top and socle. By the cyclic ordering \ref{c-ordering} of the sets $\cS$ and $\cS'$ lemma follows.
\end{proof}
\begin{remark}\label{partitionofN} All simple modules appear exactly once in the union of composition factors of all $\Delta_i$s. 
\end{remark}
The modules in the set $\bf\Delta$  of Example \ref{example} are:
\begin{align}
\Delta_1=\begin{vmatrix}
   S_2 \\
   S_3
\end{vmatrix}\hspace{20pt}
\Delta_2=\begin{vmatrix}
   S_4 \\
   S_5 \\
   S_1
\end{vmatrix}
\end{align}

The modules in the set $\bf\Delta$ are very important since they arise as building blocks for most of the syzygies, as it will be shown in Corollary \ref{topsocdelta}.
Hence, it is convenient to introduce terms: $\bf\Delta$-top, $\bf\Delta$-soc, $\bf\Delta$-submodule, $\bf\Delta$-projectives... etc.

\begin{remark} The modules ${\bf\Delta}=\{\Delta_1,\Delta_2, \dots, \Delta_r\}$  satisfy the following properties.
\begin{enumerate}[label=(\roman*)]
\item  $\{soc \Delta_i\}_{i=1}^r = \{S_{k_2}, S_{k_4},\dots, S_{k_{2r}}\}=\{soc I_{k_2}, soc I_{k_4},\dots, soc I_{k_{2r}},\}$ where \\
$\{ I_{k_2},  I_{k_4},\dots,  I_{k_{2r}}\}$  is a complete set of pairwise nonisomorphic indecomposable projective-injective modules.
\item The simple modules $top\Delta_{i+1}$ and $soc \Delta_i$  for $i=1,..,r\!-\!1$ and  $top\Delta_{1}$ and $soc \Delta_r$ are successive in the cyclic ordering of simples $(*_S)$ as in Definition \ref{orderingd}, i.e.
$$top \Delta_{i+1}\ \cdot\!\!\!\!>soc \Delta_i \text{ for } i=1,..,r\!-\!1 \text{ and } top\Delta_{1}\ \cdot\!\!\!\!>soc \Delta_r.$$
\end{enumerate}
\end{remark}

\begin{corollary} \label{Deltas}Consider the modules in ${\bf\Delta}=\{\Delta_1,..,\Delta_r\}$. Then:
\begin{enumerate}[label=(\arabic*)]
\item The cyclic ordering in Proposition \ref{cyclicorder} induces cyclic ordering on $\Delta_i$'s as 
$$\dots\Delta_{1}\ \cdot\!\!\!\!> \Delta_r \ \cdot\!\!\!\!> \Delta_{r-1}\ \cdot\!\!\!\!> \dots \ \cdot\!\!\!\!>\Delta_2\ \cdot\!\!\!\!>  \Delta_{1}\ \cdot\!\!\!\!> \Delta_r \dots$$
\item If there is a nonsplit exact sequence $0\rightarrow\Delta_j\rightarrow X\rightarrow \Delta_{j'}\rightarrow 0$, then $\Delta_j$ and $\Delta_{j'}$ must be successive in the cyclic ordering: $\Delta_j \ \cdot\!\!\!\!> \Delta_{j'}$, i.e. $j=j'+1$.
\end{enumerate}
\end{corollary}

The following is a useful lemma about the modules over cyclic Nakayama algebras.

\begin{lemma}\label{submodules} Let $\Lambda$ be a  Nakayama algebra. Let $M, N$ be indecomposable $\Lambda$-modules.
\begin{enumerate}[label=(\arabic*)]
\item  If $soc M\cong soc N$ then either $M$ is isomorphic to a submodule of $N$ or $N$ is isomorphic to a submodule of $M$.
\item If $top M\cong top N$ then either $M$ is isomorphic to a quotient of $N$ or $N$ is isomorphic to a quotient of $M$.
\end{enumerate}
\end{lemma}

\begin{proposition}{\label{topsocilk}}
 Let $M$ be an indecomposable $\Lambda$-module. Suppose  $top M\in \mathcal S'$ and $soc M\in \mathcal S$ (up to isomorphisms). Then:\\
 (a) If $M$ is of minimal length among such modules, then $M\cong \Delta_i$ for some $i$.\\
(b) The module $M$ has a filtration by the modules in $\bf\Delta$, i.e. $M$ is a $\bf\Delta$-module.
\end{proposition}
\begin{proof}
(a) Suppose $top M\cong S_{k_{2j}+1}$. Since $top \Delta_{j+1}\cong S_{k_{2j}+1}$ it follow by Lemma \ref{submodules}
that either $M$ is isomorphic to a quotient of  $\Delta_{j+1}$ or $\Delta_{j+1}$ is isomorphic to a quotient of $M$. By lemma \ref{simplesindelta}  it follows that $\Delta_{j+1}$ has only one simple from $\mathcal S$ as a composition factor. It follows that $M\cong \Delta_{j+1}$ in the first case. By minimality of $M$ the second case is also possible only if  $M\cong \Delta_{j+1}$.\\
(b) Socle of $M$ is in $\cS$, hence there exists $\Delta_i$ with the same socle. If $M\hookrightarrow\Delta_i$, by the part (a), we get $M\cong\Delta_i$. So we assume that $\Delta_i$ is a proper submodule of $M$. \\
 Let  $\Delta^j_i$ denote the module with $\bf\Delta$-socle $\Delta_i$, and $\bf\Delta$-length $j$. We can construct the maximal submodule $\Delta^{k-1}_i$ of $M$ in the sense that $\Delta^{k}_i$ is not a proper submodule of $M$. We have two cases depending on indecomposability of $\Delta^{k}_i$. Since $\Delta^{k-1}_i$ is submodule of $M$, it is immediately indecomposable by Lemma \ref{submodules}. 
Assume that $\Delta^k_i$ is indecomposable. By the induced  cyclic ordering to $\bf\Delta$-modules in \ref{Deltas}, either $top M\cong top\Delta^{k-1}_i$ or $top M\cong\Delta^k_i$. Hence it has $\bf\Delta$-filtration.
In the other case, $\Delta^k_i$ splits into direct sum:
\begin{align}
\Delta^{k-1}_i\hookrightarrow M\hookrightarrow \Delta^{k-1}_i\oplus \Delta_{i'}
\end{align}
for some index $i'$. The only possibility is $M\cong \Delta^{k-1}_j$.
\end{proof}

\begin{lemma}\label{folklore1}
(a) A projective module $P$ cannot be a proper subquotient of another indecomposable module $X$.
(b) Moreover, if $\Lambda$ is Nakayama, a projective module cannot be a submodule of a nonprojective indecomposable modules.
\end{lemma}

\begin{proof}
A projective module at vertex $i$ has the following characterization: it is the longest possible path starting from $i$. Assume that $P=P_i$ is subquotient of another projective $P'$, however this immediately implies existence of longer path ending at index of $socP$ by keeping in mind that all modules are uniserial. By the same argument, a projective module can be a submodule of only a projective module.\\
It is possible to have two projectives which share the same socle.
\end{proof}

\begin{lemma}\label{folklore2}
Let $\Lambda$ be a cyclic Nakayama algebra of infinite global dimension. Assume that the module $M$ has infinite projective dimension. Then there exists $k$ such that $\Omega^i(M)$ is $\Omega$-periodic for all $i\geq k$. 
\end{lemma}

\begin{proof}
Recall that there are only finitely many indecomposible modules for $\Lambda$ and all of them are finite dimensional. Furthermore syzygies are indecomposable. 
Therefore, there exist $i\neq j$ such that $\Omega^i(M)\cong\Omega^j(M)$. Suppose $j>i$. Then $\Omega^{j-i}(\Omega^i(M))=\Omega^j(M)\cong \Omega^i(M)$. Therefore $\Omega^i(M)$ is $\Omega$-periodic.
\end{proof}
\begin{lemma}\label{folklor3} Let $0\rightarrow A\rightarrow B\rightarrow C\rightarrow 0$ be a nonsplit exact sequence of uniserial indecomposable modules.
\begin{enumerate}[label=(\arabic*)]
\item Then $soc A\cong soc B$ and $top B\cong top C$.
\item If $A, B,C$ are $\bf\Delta$-modules then  ${\bf\Delta}$-$soc A\cong{\bf\Delta}$-$soc B$ and ${\bf\Delta}$-$top B\cong {\bf\Delta}$-$top C$.
\end{enumerate}
\end{lemma}


\begin{corollary}\label{topsoc} Let $\Lambda$ be a cyclic Nakayama algebra and let $0\rightarrow A\rightarrow B\rightarrow C\rightarrow 0$ be a nonsplit exact sequence of uniserial indecomposable modules. 
\begin{enumerate}[label=(\arabic*)]
\item Then $top A$ and $soc C$  are consecutive in the cyclic ordering of simples, i.e.  $top\,A\ \cdot\!\!\!\!>soc\,C$ up to isomorphism.
\item  If $A, B,C$ are $\bf\Delta$-modules then ${\bf\Delta}$-$top\,A$ and ${\bf\Delta}$-$soc\,C$  are consecutive in the cyclic ordering of ${\bf\Delta}$-simples, i.e.  ${\bf\Delta}$-$top\,A\ \cdot\!\!\!\!>{\bf\Delta}$-$soc\,C$ up to isomorphism.
\end{enumerate}
\end{corollary}
\begin{proof}
(1) Since the module $B$ is uniserial and indecomposable, all composition factors of $B$  appear in the cyclic order of $(*_S)$. 

Without loss of generality assume that $C$ has the following composition factors 
\begin{align}
C=\begin{vmatrix}
   S_1 \\
   \vdots \\
   S_z
\end{vmatrix}
\end{align}
Since $B$ is indecomposable uniserial
and $A$ is nontrivial module, all composition factors of $C$ have to appear in $B$. Moreover, $S_{z+1}$ has to be a composition factor of $B$. (Otherwise  A is trivial and $C\cong B$). By the exactness of the sequence, $B/A\cong C$ hence  $top A\cong S_{z+1}$, which is consecutive to $S_z$. 

(2) Using the same argument on $\bf\Delta$-modules and the corollary \ref{Deltas}, the claim follows.
\end{proof}

We recall that $N$ and $1$ is also consecutive.
We can see it on the Example \ref{example}:
\begin{align}
0\rightarrow \begin{vmatrix}
   S_3
   \end{vmatrix} \rightarrow \begin{vmatrix}
   S_1 \\
   S_2 \\
   S_3
\end{vmatrix} \rightarrow \begin{vmatrix}
   S_1 \\
   S_2
\end{vmatrix} \rightarrow 0\\
0\rightarrow \begin{vmatrix}
   S_1\\
   S_2\\
   S_3
   \end{vmatrix} \rightarrow \begin{vmatrix}
	S_4\\   
   S_5\\
   S_1 \\
   S_2 \\
   S_3
\end{vmatrix} \rightarrow \begin{vmatrix}
   S_4 \\
   S_5
\end{vmatrix} \rightarrow 0
\end{align}
The simple modules $S_5$ and $S_1$ are consecutive, and similarly $S_2$ and $S_3$.

\begin{remark}Let $\Lambda$ be a Nakayama algebra. Then:
\begin{enumerate}[label=(\roman*)]
\item All indecomposable modules are uniserial.
\item All non-zero syzygies of indecomposable modules are indecomposable.
\end{enumerate}
\end{remark}

\begin{proposition}\label{deltafilt} 
Let $X$ be an indecomposable $\Lambda$-module. Then:
\begin{enumerate}[label=(\roman*)]
\item For $\forall i\geq 1$, if $\Omega^i(X)\neq0$ then $soc \,\Omega^i(X)\in \mathcal S$ up to isomorphism.
\item For $\forall i\geq 2$, if $\Omega^i(X)\neq0$ then $top\,\Omega^i(X)\in \mathcal S'$ up to isomorphism.
\end{enumerate}
\end{proposition}

\begin{proof}
(i) Consider the first few steps of the projective resolution of the module $X$:
\[\xymatrixcolsep{10pt}
\xymatrix{
\ar[r] & P\Omega(X)\ar[rd] \ar[rr] && PX \ar[rr] && X  \\
\Omega^2(X)\ar[ru]& &\Omega(X) \ar[ru] & & 
}\]
Observe that $\Omega^i(X)$ is a submodule of $P(\Omega^{i-1}X)$, the  projective cover of $\Omega^{i-1}(X)$. Hence 
$soc\,\Omega^i(X)=soc\,P(\Omega^{i-1}(X))$ for $\forall i\geq 1$. By the characterization of projective modules of $\Lambda$ from Section \ref{projectives}, $soc\,\Omega^i(X)\in \cS$ for $\forall i\geq 1$ and when $\Omega^i(X)\neq0$. 

(ii) Consider the following short exact non-split sequences for $\forall i\geq 2$:
$$0\rightarrow \Omega^i(X)\rightarrow P(\Omega^{i-1}(X))\rightarrow \Omega^{i-1}(X)\rightarrow 0.$$
Since $i\geq 2$, it follows from (i) that $soc\,\Omega^{i-1}(X)\in \mathcal S$. Hence $soc\,\Omega^{i-1}(X)$ is $S_{k_{2j}}$ for some $j\in\{1,..,r\}$. Then by Corollary \ref{topsoc} it follows that $top\,\Omega^i(M)$ is $S_{k_{2j}+1}$, hence $top\,\Omega^i(X)\in\mathcal S'$ up to isomorphism.
\end{proof}


\begin{corollary} \label{topsocdelta}Let $X$ be an indecomposable $\Lambda$-module. Then for $\forall i\geq 2$, the modules $\Omega^i(X)$ are filtered by the modules in $\bf\Delta$, i.e. are $\bf\Delta$-modules.
\end{corollary} 
\begin{proof}Since socles and tops of modules appearing as summands in $\Omega^i(X)$, $i\geq 2$ are in $\cS$ and $\cS'$ respectively, by Proposition \ref{topsocilk}, all these modules have $\bf\Delta$-filtration.
\end{proof}


\begin{theorem}
If the set ${\bf\Delta}=\{\Delta_1,..,\Delta_r\}$ contains a projective module, then the global dimension of $\Lambda$ is finite.
\end{theorem}
\begin{proof}
Without loss of generality, let $\Delta_r$ be a projective module. 
 Let $M$ be an indecomposable $\Lambda$-module. We want to show that $\pdim M <\infty$. If $\pdim M \leq 1$ we are done. Assume $\pdim M \geq 2$. Then $\Omega^2(M)\neq 0$. By  the Corollary \ref{topsocdelta} it follows that $\Omega^2(M)$ is filtered by the modules in $\bf\Delta$. So it is enough to show that all $\bf\Delta$-modules have finite projective dimension. 

Claim: Let  $D$ be a $\bf\Delta$-module with $\bf\Delta$-top isomorphic to $\Delta_{r-d}$.  Then $\pdim D\leq d$. \\
The proof will be by induction on $d$. It is clear that $0\leq d\leq r-1$. For $d=0$, notice that any module which has $\Delta_r$ as the $\bf\Delta$-top must be isomorphic to $\Delta_r$ since $\Delta_r$ is projective.

Assume that the claim is true for $d=m-1$. We need to analyze the case $d=m$. Consider a module $D'$ with $\bf\Delta$-top isomorphic to $\Delta_{r-m}$. Let  $\bf\Delta$-soc of $D'$ be isomorphic to $\Delta_{r-t}$, for some $t$ such that $2\leq t\leq m$. Then the short exact sequence
$$0\rightarrow \Omega(D')\rightarrow P(D')\rightarrow D'\rightarrow 0$$
has a more precise description in terms of  $\bf\Delta$-filtrations by using Corollary \ref{topsoc} :
\begin{align}
0\rightarrow \begin{matrix}
\Delta_{r-(t-1)}\\
\vdots\\
{\bf\Delta} socP_{\Delta_{r-m}}\end{matrix}\rightarrow P_{\Delta_{r-m}}\rightarrow 
\begin{matrix}
\Delta_{r-m}\\
\vdots\\
\Delta_{r-t}
\end{matrix}\rightarrow 0.
\end{align}
Then $\pdim D'=1+\pdim\Omega(D')$. Since $\bf\Delta$-top of $\Omega(D')$ is isomorphic to $\Delta_{r-(t-1)}$, it follows by induction that $\pdim \Omega(D') \leq t-1$ which implies that 
\begin{align}
\pdim D'\leq t\leq m
\end{align}
In particular, each module in ${\bf\Delta}=\{\Delta_1,..,\Delta_r\}$ has finite projective dimension and therefore any extension of these modules has finite projective dimension. As a result, $\Lambda$ is of finite global dimension.
\end{proof}

\section{The $\varphi$-dimension} 

We now recall the definition of the function $\varphi$ for modules over any artin algebra $\Lambda$, and also give a particularly nice description of this function for Nakayama algebras.

\subsection{Artin algebras and their $\varphi$-dimension}
In order to define $\varphi(M)$ for each module $M$ we need the following set-up:
let $K_0$ be the abelian group generated by all symbols $\left[X\right]$, where $X$ is a finitely generated $\Lambda$-module, modulo the relations:
\begin{enumerate}[label=\roman*)]
\item $\left[C\right]=\left[A\right]+\left[B\right]$ if $C\cong A\oplus B$
\item $\left[P\right]=0$ if $P$ is projective.
\end{enumerate}
Then $K_0$ is the free abelian group generated by the isomorphism classes of indecomposable finitely generated nonprojective $\Lambda$-modules. For any finitely generated $\Lambda$-module $M$ let $L\left[M\right]:=\left[\Omega M\right]$ where $\Omega M$ is the first syzygy of $M$. Since $\Omega$ commutes with direct sums and takes projective modules to zero this gives a homomorphism $L : K_0 \mapsto K_0$. For every finitely generated $\Lambda$-module $M$ let $\left\langle addM\right\rangle$ denote the subgroup of $K_0$ generated by all the indecomposable summands of $M$, which is a free abelian group since it is a subgroup of the free abelian group $K_0$.  
In other words, if $M=\oplus_{i=1}^m{M_i}^{n_i}$ then $\left\langle addM\right\rangle=\left\langle \{[M_i]\}_{i=1}^m\right\rangle$ and $L^t(\left\langle addM\right\rangle)= \left\langle \{[\Omega^t M_i]\}_{i=1}^m\right\rangle$.
\begin{definition}\label{varphi}
For a given module $M$, let $\varphi\left(M\right)$ be defined as:
$$\varphi(M):=min\{t\ |\ rank\left(L^t\langle add M\rangle\right)=rank\left(L^{t+j}\langle add M\rangle\right)\text{ for }\forall j\geq 1\}.$$
\end{definition}

\begin{remark}\label{phi} Let $\varphi$ be the function defined above. Then,
\begin{enumerate}[label=(\arabic*)]
\item  
$\varphi(M)$ is finite for each module $M$.
\item Suppose the projective dimension $\pdim(M)<\infty$. Then $\varphi(M)=\pdim(M)$.
\item The function $\varphi$ is not additive, hence it is not enough to compute it on indecomposable modules.
\item Suppose $M$ is indecomposable with $\pdim(M)=\infty$. Then  $\varphi(M)=0$.
\end{enumerate}
\end{remark}

\begin{definition}\label{Lambda} Let $\Lambda$ be an artin algebra. We define:
$$\varphi\dim(\Lambda):=sup\{\varphi(M)\ |\ M \in \text{mod-}\Lambda\}.$$ 
\end{definition}

In order to determine $\varphi\dim(\Lambda)$ for Nakayama algebras we will first state an easy lemma which holds in general and then apply it to the algebras of finite representation type and then to the particular situation of Nakayama algebras.

\begin{lemma} Let $\Lambda$ be an artin algebra. Let $M$ and $N$ be  $\Lambda$-modules. Then:
\begin{enumerate}[label=(\arabic*)]
\item $\varphi (M)\leq \varphi(M\oplus N)$,
\item $\varphi (M^n)= \varphi(M)$ for all $n\geq 1$.
\end{enumerate}
\end{lemma}

\begin{lemma} \label{M} Let $\Lambda$ be an artin algebra of finite representation type. Let $\{M_1,\dots,M_m\}$ be a complete set of representatives of isomorphism classes of indecomposable $\Lambda$-modules. Let $M^{\dagger}:=\oplus_{i=1}^mM_i$. Then 
$\varphi \dim(\Lambda)=\varphi(M^{\dagger})$.
\end{lemma}

\begin{proof} Let $X$ be in mod-$\Lambda$. Then $X\cong\oplus_{i=1}^mM_i^{n_i}$. Therefore $X$ is isomorphic to a summand of $(M^{\dagger})^k$ for a sufficiently large $k$. Then by the above lemma 
$\varphi (X)\leq \varphi ((M^{\dagger})^{k})=\varphi (M^{\dagger})$. Hence $\varphi \dim(\Lambda) \leq \varphi(M^{\dagger})$ and therefore $\varphi(M^{\dagger})= \varphi \dim(\Lambda)$.
\end{proof}
\subsection{Nakayama algebras and their $\varphi$-dimension}
In order to study $\varphi\dim (\Lambda)$ for Nakayama algebras the following notion will be useful: Let $X$ be a $\Lambda$-module. Let: 
$$\alpha(X)=\#\{\text{pairwise non-isomorphic, non-projective, indecomposable summands of } X\}.$$ 

\begin{lemma}\label{alpha} Let $\Lambda$ be a Nakayama algebra. Let $X=\oplus X_i$ be a $\Lambda$-module. Then:
\begin{enumerate}[label=(\roman*)]
\item $\alpha(\Omega X)\leq\alpha(X)$
\item $\alpha(\Omega X)=\alpha(X)$ if and only if (1) $\Omega X_i$ is non-projective for all non-projective summands $X_i$ of $X$ and (2) $\Omega X_i\not\cong \Omega X_j$ for $X_i\not\cong X_j$.
\end{enumerate}
\end{lemma}
\begin{proof}Notice that $\Omega(Y)$ is either $0$ or indecomposable for all indecomposable modules $Y$.
\end{proof}
\begin{lemma} \label{phiNakayama} Let $\Lambda$ be a Nakayama algebra.  Let $\{M_1,\dots,M_m\}$ be a complete set of representatives of isomorphism classes of indecomposable $\Lambda$-modules and let \\
$M^{\dagger}:=\oplus_{i=1}^mM_i$. Then:
$$\varphi\dim(\Lambda)=min\{t\ |\ \alpha(\Omega^t(M^{\dagger}))=\alpha(\Omega^{t+j}(M^{\dagger})), \  \forall j\geq 0\}.$$
\end{lemma} 

\begin{proof} By Lemma \ref{M} and Definition \ref{varphi} it follows that $\varphi\dim(\Lambda)= \varphi(M^{\dagger})=$\\
$=min\{t\ |\ rank\left(L^t\langle add M^{\dagger}\rangle\right)=rank\left(L^{t+j}\langle add M^{\dagger}\rangle\right), \  \forall j\geq 1\}$=\\
$=min\{t\ |\ rank \left\langle \{[\Omega^t M_i]\}_{i=1}^m\right\rangle =rank \left\langle \{[\Omega^{t+j} M_i]\}_{i=1}^m\right\rangle, \  \forall j\geq 1\}$=\\
$=min\{t\ |\ \alpha(\Omega^t(M))=\alpha(\Omega^{t+j}(M)),\  \forall j\geq 1\}$. The last equality follows since $\Lambda$ is Nakayama and the syzygies of indecomposable modules are always indecomposable.
\end{proof}
\begin{remark}\label{fidimfindim} Recall that $\findim \Lambda = sup\{\pdim (X)\ |\ \pdim (X)< \infty\}$. Then $\findim \Lambda\leq \varphi\dim\Lambda$.
\end{remark}


\begin{definition}\label{folklor4} A module $X$ is called $\Omega$-periodic if there exists  an integer $k\geq 1$ such that $\Omega^{k}(X)\cong X$. After relabeling the modules $\{M_1,\dots,M_m\}$ let $\{M_1,\dots,M_a\}$ be the subset  which consists of all  $\Omega$-periodic modules. Let  $${\bf\Omega}^{per}:=\{M_1,\dots,M_a\}\ \  and \ \ M^{per}:=\oplus_{i=1}^aM_i.$$
\end{definition}

\begin{lemma} \label{periodic} If $X$ is $\Omega$-periodic then $\Omega^i(X)$ is also $\Omega$-periodic for all $i\geq 1$.
\end{lemma}
\begin{proof} Suppose $\Omega^k(X)\cong X$. Then $\Omega^{k}(\Omega^i(X))=\Omega^{i}(\Omega^{k}(X))\cong \Omega^{i}(X)$. 
\end{proof}

\begin{lemma}\label{hakem}  Let $\Lambda$ be Nakayama algebra and let ${\bf\Omega}^{per}=\{M_1,\dots,M_a\}$. Then the syzygy operation defines a bijection $\omega: {\bf\Omega}^{per}=\{M_1,\dots,M_a\}\to{\bf\Omega}^{per}=\{M_1,\dots,M_a\}$.
\end{lemma}
\begin{proof} 
By the above lemma, $\Omega(M_j)$ is again an $\Omega$-periodic module, hence is isomorphic to one of the elements in ${\bf\Omega}^{per}=\{M_1,\dots,M_a\}$. Define $\omega(M_j)$ to be that element of ${\bf\Omega}^{per}$, hence 
$\omega(M_j)=M_{j'}$ for some $j'\in\{1,\dots,a\}$. To show that this is a bijection, it is enough to show that $\omega$ is surjective since the set is finite.
Let $M_p$ be an element in ${\bf\Omega}^{per}$ and let $\Omega^{k_p}(M_p)\cong M_p$. Then 
$\Omega(\Omega^{k_p-1}(M_p))\cong M_p$.  By Lemma \ref{periodic} the module $\Omega^{k_p-1}(M_p)$ is periodic and therefore isomorphic to some $M_q\in {\bf\Omega}^{per}$.  Therefore $\omega(M_q)=M_p$.
\end{proof}

\begin{proposition} \label{nonperiodic} Let $\Lambda$ be a Nakayama algebra. Let $k\geq \findim \Lambda$. Suppose $\Omega^k(M)=M^{per}\oplus N$ with $N\neq 0$. Then $\alpha (\Omega^k(M))>\alpha (\Omega^{\varphi\dim\Lambda}(M))$ and $\varphi\dim\Lambda\geq k$.
\end{proposition}
\subsection{Periodic syzygies for Nakayama algebras}
 
Since syzygies become "stable" at $\varphi\dim(\Lambda)$,  we need to describe the module $\Omega^{\varphi\dim(\Lambda)}(M^{\dagger})$ where $M^{\dagger}=\oplus_{i=1}^mM_i$.

\begin{proposition}Let $\Lambda$ be Nakayama algebra. Then $\Omega^{\varphi\dim(\Lambda^{\dagger})}(M)\cong M^{per}$ where $M^{per}$ is the direct sum of the complete set of representatives of isomorphism classes of $\Omega$-periodic modules.
\end{proposition}
\begin{proof} It follows from Proposition \ref{nonperiodic} that if $\Omega^t(M^{\dagger})$ has a summand which is not $\Omega$-periodic, then $\varphi\dim\Lambda>t$. To see that all representatives of $\Omega$-periodic modules are summands of $\Omega^{\varphi\dim(\Lambda)}$, notice that $M^{per}$ is direct summand of $M$ and of every $\Omega^t(M^{\dagger})$ for all $t\geq 0$.
\end{proof}

\begin{definition}\label{mydefinition} let $\Lambda$ be Nakayama algebra. We define:
$\cP^{per}$ is the complete set of representatives of isomorphism classes of projective covers of modules in ${\bf\Omega}^{per}$.
\end{definition}

\begin{remark} All projective modules in $\mathcal P^{per}$ have filtration by the modules from ${\bf\Delta}$, i.e. are
${\bf\Delta}$-modules since the tops of such projectives are in $\mathcal S'$ and socles are in $\mathcal S$ and by Proposition \ref{topsoc} they are ${\bf\Delta}$-modules.
\end{remark}


\subsection{Self-injective Nakayama Algebras}
We list some useful characterizations of self injective cyclic Nakayama algebras.
\begin{remark} \label{self-injective} Let $\Lambda$ be a cyclic Nakayama algebra. The followings are equivalent:\label{selfinjective}
\begin{enumerate}[label=(\arabic*)]
\item $\Lambda$ is self-injective algebra.
\item All projective modules are injective.
\item All projectives have the same length. 
\item All radicals of projectives have the same length. 
\item Non-isomorphic projectives have non-isomorphic socles.
\item Each projective is a minimal projective. 
\item Each projective is in a different class.
\item Number of classes in the notation \ref{classes} is equal to the number of projectives.
\end{enumerate}
\end{remark}

\subsection{Small values of $\varphi$-dimension}
For the small values of $\varphi$-dimension we have the following theorem
where the first result appears in both \cite{lanz} and \cite{ralf} for general artin algebras, but for the Nakayama algebras we give an alternative short proof.

\begin{theorem}\label{thm1}
Let $\Lambda$ be a cyclic Nakayama algebra of infinite global dimension. Then we have the following:
\begin{enumerate}[label=(\roman*)]
\item $\varphi\dim\Lambda=0$ $\iff$ $\Lambda$ is self-injective algebra. 
\item $\varphi\dim\Lambda\neq 1$ 
\item $\varphi\dim\Lambda=2$ $\iff$ ${\bf\Delta}=\{\Delta_1,\dots\Delta_r\}\subseteq {\bf\Omega}^{per}$
\end{enumerate}
\end{theorem}
\begin{proof} (i)
$(\Leftarrow)$
Assume that $\Lambda$ is self-injective Nakayama algebra. In order to show $\varphi\dim\Lambda = 0$ it is enough to show that $\alpha(M^{\dagger})=\alpha(\Omega(M^{\dagger}))$ where $M^{\dagger}:=\oplus_{i=1}^mM_i$ by lemma \ref{phiNakayama}. So, it is enough to show that every non-projective indecomposable module $M_i$ is isomorphic to $\Omega(M_j)$ for some indecomposable module $M_j$. Let $M_i\to I$ be the injective envelope. Then $I$ is also projective since $\Lambda$ is self-injective. Therefore
$M_i\cong\Omega(I/M_i)$. Notice $I/M_i\neq 0$, since $M_i$ is not projective.


$(\Rightarrow)$
We start with $\varphi\dim\Lambda=0$. This implies that $\alpha(M^{\dagger})=\alpha(\Omega(M^{\dagger}))$. In order to show that $\Lambda$ is self-injective, it is enough to show that non-isomorphic projectives have non-isomorhic socles by Remark \ref{selfinjective}. Suppose $P_x$ and $P_y$  have the same socle. Then, by Lemma \ref{submodules} one of these modules is isomorphic to a submodule of the other. Say $P_y$ is submodule of $P_x$. Then $\Omega(P_x/P_y)\cong P_y$ and by Lemma \ref{alpha} 
it follows that $\alpha(M^{\dagger})>\alpha(\Omega(M^{\dagger}))$ which gives contradiction to the assumption that $\alpha(M^{\dagger})=\alpha(\Omega(M^{\dagger}))$.

(ii) Summands of module $\Omega(M^{\dagger})$ can have any simple module as a top. However tops of summands of module $\Omega^2(M^{\dagger})$ have to be in $\cS$ by proposition \ref{deltafilt}. This implies existence of a summand of $\Omega(M^{\dagger})$ which is not a summand of $\Omega^2(M^{\dagger})$. Therefore $\alpha(\Omega^2(M^{\dagger}))<\alpha(\Omega(M^{\dagger}))$.

(iii) Before proving the theorem, we need some auxiliary results. Recall that a projective module is called $\bf\Delta$-projective if it has $\bf\Delta$-filtration. Let ${\bf\Delta}$P be the complete set of 
representatives of isomorphism classes of projective modules with $\bf\Delta$-filtration.

\begin{lemma}\label{newproj1} All $\bf\Delta$-projective modules are of the same $\bf\Delta$-length if and only if $\bf\Delta$ is a subset of $\Omega^{per}$.
\end{lemma}

\begin{proof}
 We claim that projectives in ${\bf\Delta}$P have the same $\bf\Delta$-length if and only if $\bf\Delta$ is subset of ${\bf\Omega}^{per}$. Suppose that all projective modules in ${\bf\Delta}$P have the same $\bf\Delta$-length. We consider the resolution of $\Delta_i$ for each $i$. Since each projective cover of $\Delta_i$ has the same $\bf\Delta$-length, this implies that $\bf\Delta$-length of syzygies of $\Delta_i$'s is one less then length of their projective covers. Now, if we compute their second syzygies, it turns out that their $\bf\Delta$-length is one. So for any even indexed syzygy,  all $\Delta_i$'s appear, which implies that $\bf\Delta$ is subset of ${\bf\Omega}^{per}$. For other direction, assume that $\bf\Delta$ is subset of ${\bf\Omega}^{per}$. This implies that $\cP^{per}$ contains all $\bf\Delta$-projective modules i.e. ${\bf\Delta}$P$\subset\cP^{per}$. Moreover, since each $\bf\Delta$-module is in ${\bf\Omega}^{per}$, they have to appear as $\bf\Delta$-soc of each projective module in $\cP^{per}$. Otherwise, if one $\Delta_i$ is  not a submodule of a ${\bf\Delta}$-projective then $\Delta_i\notin{\bf\Omega}^{per}$. Therefore we get $\cP^{per}={\bf\Delta}$P. In this case, this forces that projective modules in $\cP^{per}$ are characterized by not only their $\bf\Delta$-tops but also their $\bf\Delta$-socles.  Hence we obtain an analogue of self-injective algebra with respect to $\bf\Delta$-filtration.
\end{proof}

\begin{corollary}\label{newcor1} If all projective modules in  ${\bf\Delta}$P have the same $\bf\Delta$-length, we have:
\begin{enumerate}[label=(\roman*)]
\item then projective dimension of any $\bf\Delta$-module is infinite.
\item if $\Omega(X)\cong\Omega(Y)$ then $X\cong Y$ where $X,Y$ are ${\bf\Delta}$-modules. 
\end{enumerate}
\end{corollary}
\begin{proof} (i) Let $l(m)$ be the $\bf\Delta$-length of an indecomposable $\bf\Delta$-module $M$, and $l(p)$ be the length of any projective of $\cP^{per}$. Now the length of $\Omega^t(M)$ is either $l(m)$ or $l(p)-l(m)$. In both cases it is smaller then $l(p)$, hence none of the syzygies are projective.\\
(ii) Assume that $\Omega(X)\cong\Omega(Y)$. This implies that the projective modules which have them as submodules are isomorphic. Since socles of those projectives are isomorphic and moreover they have the same length. This implies that $X\cong Y$.
\end{proof}

 Now we can use the above arguments to prove our theorem. Assume that ${\bf\Delta}\subseteq{\bf\Omega}^{per}$. By the lemma \ref{newproj1}, all ${\bf\Delta}$-projectives  have the same length. By proposition \ref{topsocilk}, every summand of $\Omega^2(M^{\dagger})$ has $\bf\Delta$-filtration. Moreover each $\Delta_i$ is summand of $M^{per}$ by the assumption. Now, we get $\Omega^2(M^{\dagger})\cong M^{per}$. Since any indecomposable summand of $\Omega^2(M^{\dagger})$ is a periodic syzygy by corollary \ref{newcor1}.
 
 For the other direction, we start with ${\bf\Omega}^2={\bf\Omega}^{per}$. Since $\bf\Delta$ is always a subset of ${\bf\Omega}^2$ by proposition \ref{deltafilt}, we get ${\bf\Delta}\subseteq{\bf \Omega}^{per}$. 
\end{proof}

We want to illustrate the theorem on the example \ref{example}. The set $\{M_1,\dots,M_m\}$ consists of $16$ elements, one can count them in the Auslander-Reiten quiver. A module can appear in $\Omega^1$ and $\Omega^2$ if it is a nonprojective submodule of a projective. Hence we get
\begin{gather*}
\Omega(M^{\dagger})\cong \begin{vmatrix}
   S_1
\end{vmatrix}\oplus\begin{vmatrix}
   S_5 \\
   S_1
\end{vmatrix}\oplus\begin{vmatrix}
   S_4 \\
   S_5 \\
   S_1
\end{vmatrix}\oplus\begin{vmatrix}
   S_3 
\end{vmatrix}\oplus\begin{vmatrix}
   S_2 \\
   S_3
\end{vmatrix}\\
\Omega^2(M^{\dagger})\cong\begin{vmatrix}
   S_2 \\
   S_3
\end{vmatrix}\oplus\begin{vmatrix}
   S_4 \\
   S_5  \\
   S_1
\end{vmatrix}\cong\Delta_1\oplus\Delta_2
\end{gather*}

Since $\Omega(\Delta_1)\cong\Delta_{2}$ and $\Omega(\Delta_2)\cong\Delta_1$ we conclude that ${\bf\Omega}^{per}={\bf\Omega}^2$ and under $\bf\Delta$- filtration, projectives are 
\begin{align}
P_2=\begin{vmatrix}
   \Delta_1 \\
   \Delta_2
\end{vmatrix},\hspace{10pt} P_4=\begin{vmatrix}
   \Delta_2 \\
   \Delta_1
\end{vmatrix}
\end{align}
Clearly they have the same ${\bf\Delta}$-length.

\section{Even $\varphi$-dimension}

We will prove 
that the $\varphi$-dimension of Nakayama algebras of infinite global dimension is always an even number. To do this, we need to develop a few techniques.
Again, we suppose that global dimension of $\Lambda$ is infinite. This implies that there are two types of modules either their resolution stops at a projective module i.e. $\pdim M<\infty$ or reaches ${\bf\Omega}^{per}$ i.e. $\pdim M=\infty$. 

\subsection{Properties of periodic syzygy ${\bf\Omega}^{per}$}
We defined in \ref{folklor4} the set ${\bf\Omega}^{per}=\{M_1,\dots, M_a\}$ where $\{M_1,\dots, M_a\}$ is a complete set of representatives of the indecomposable pairwise non-isomorphic $\Omega$-periodic modules. Some of its properties were given in the lemmas \ref{periodic}, \ref{hakem}.

For the analysis of $\varphi$-dimension it will be important to analyze when syzygy of a module will be in ${\bf\Omega}^{per}$, so we introduce the following notion.

\begin{definition}\label{rho} Let $X$ be an indecomposable $\Lambda$-module of infinite projective dimension. We define:
$$\rho(X):=min\{ t\ |\ \Omega^t(X)\in {\bf\Omega}^{per},\, t\geq 0 \}.$$
\end{definition}
We can extend $\rho$ into direct sums also.
\subsection{Modification of projective resolutions when $\rho(X)$ is odd}
The following is an important lemma for the rest of the paper, we will use it repeatedly in order to modify given projective resolutions into projective resolutions with desired properties.

\begin{lemma}\label{presiz}
Suppose that $X\in{\bf\Omega}^{per}$. Then there exists a projective module $P$ in $\cP^{per}$ such that $X$ is isomorphic to a submodule of $P$. In other words, every periodic module is a submodule of a periodic projective.
\end{lemma}

\begin{proof}
By the definition \ref{folklor4}, we know that there exist a $p$ such that $\Omega^p(X)\cong X$, and $\Omega^{p-1}(X)\in{\bf\Omega}^{per}$.
Since $\Omega^{p-1}(X)$ is an element of ${\bf\Omega}^{per}$, its projective cover $P$ is in $\cP^{per}$. Then $P$ is nontrivial extension of $\Omega^{p-1}(X)$ and $\Omega^p(X)$:
\begin{align}
0\rightarrow \Omega^p(X)\rightarrow P \rightarrow \Omega^{p-1}(X)\rightarrow 0.
\end{align}
Hence $\Omega^p(X)$ is a proper submodule of $P$. Since we have $X\cong\Omega^p(X)$, this implies that $X$ is isomorphic to a proper submodule of $P$.
\end{proof}

\begin{corollary}\label{cor1}
For any $\Delta_i$ which appears as $\bf\Delta$-socle of a module $X\in{\bf\Omega}^{per}$, there exists a projective $P\in\cP^{per}$ such that $\bf\Delta$-soc$P\cong\Delta_i$.
\end{corollary}
\begin{proof}
By the previous lemma \ref{presiz}, for any module $X$ appearing in ${\bf\Omega}^{per}$, there exists a projective $P\in\cP^{per}$ so that $X$ is a submodule $P$. Since $\Delta_i$ is submodule of $X$ by the hypothesis of corollary, this implies that $\Delta_i$ is submodule of $P$.
\end{proof}

The following proposition and the remark will be used in the proof of propositions \ref{modificationofodd} \ref{tip1}, \ref{tip2}, \ref{tip3}, and then in the proof of the Theorem (A) \ref{even}. That's why we refer them as the main techniques.

\begin{proposition}\label{temelteknik} Let $0\rightarrow X\rightarrow P \rightarrow Y \rightarrow 0$ be a non split exact sequence where $X$ is a quotient of a projective $P'\in \cP^{per}$ and the module $P$ is projective.
Then, there exists a module $\tilde{P}\in\cP^{per}$ such that $Y$ is  proper submodule of $\tilde{P}$.
\end{proposition}

\begin{proof}
Since $P'\in\cP^{gen}$, there exists a module $M\in\Omega^{per}$ such that $P'$ is projective cover of $M$ by definition \ref{mydefinition}. In particular $top M\cong top X$. Moreover, there exists module $\tilde{M}\in\Omega^{per}$ such that $soc\tilde{M}\cong soc Y$ by corollary \ref{topsoc}. Because every periodic module is a submodule of a periodic projective from $\cP^{gen}$ by lemma \ref{presiz}, and modules are uniserial by  \ref{submodules}, these guarantee the existence of $\tilde{P}$.
\end{proof}

\begin{remark}\label{teknik1}
Assume that the exact sequence $0\rightarrow X\rightarrow P\rightarrow Y\rightarrow 0$ satisfies the conditions in proposition \ref{temelteknik} i.e. $P$ is projective module and $X$ is a quotient of a periodic projective. As a result, first we get the short exact sequence:
\begin{align}
0\rightarrow Y\rightarrow \tilde{P}\rightarrow \faktor{\tilde{P}}{Y}\rightarrow 0
\end{align}
and then by using it:
\begin{align}\label{4stepres}
0\rightarrow X\rightarrow P\rightarrow \tilde{P}\rightarrow\faktor{\tilde{P}}{Y}\rightarrow 0
\end{align}
which is exact.
\end{remark}

If we apply corollary \ref{topsoc} to the previous resolution (\ref{4stepres}), we obtain:
\begin{corollary}\label{topsoctriangle}
\begin{enumerate}[label=(\arabic*)]
\item $socX\cong socP$, $socY\cong soc\tilde{P}$
\item $topP\cong topY$, $top\tilde{P}\cong top\faktor{\tilde{P}}{Y}$
\item $topX$ is consecutive with $socY$
\item $topY$ is consecutive with $soc\faktor{\tilde{P}}{Y}$
\end{enumerate}
\end{corollary}

 We start with the following observation:

\begin{proposition}\label{modificationofodd}
Assume that there exists an indecomposable module $M$ such that:
\begin{enumerate}[label=\roman*)]
\item $\Omega^k\left(M\right)\in{\bf\Omega}^{per}$
\item $\Omega^{k-1}\left(M\right)\notin{\bf\Omega}^{per}$.
\end{enumerate}
for an odd integer $k$. Then, there exists a module $\tilde{M}$ such that $\Omega^{k+1}(\tilde{M})\in{\bf\Omega}^{per}$ and $\Omega^{k}(\tilde{M})\notin{\bf\Omega}^{per}$.
\end{proposition}

\begin{proof}
Let $\ldots\rightarrow P_k\rightarrow P_{k-1}\rightarrow\dots P_1\rightarrow P_0\rightarrow M\rightarrow 0$ be the projective resolution  of $M$.
Consider the short exact sequence:
\begin{align}
0\rightarrow \Omega^{k}(M)\rightarrow P_{k-1}\rightarrow \Omega^{k-1}(M)\rightarrow 0
\end{align}
Since $\Omega^{k}(M)$ is the quotient of projective $P_{k}\in\cP^{per}$, proposition \ref{temelteknik} can be applied to get the sequence of the form \ref{4stepres}:
\begin{align}
0\rightarrow \Omega^{k}(M)\rightarrow P_{k-1}\rightarrow P'_{k-2}\rightarrow\tilde{\Omega}^{k-2}(M)\rightarrow 0
\end{align}
where $P'_{k-2}\in\cP^{per}$ and $\tilde{\Omega}^{k-2}(M)$ denotes the quotient $\faktor{P'_{k-2}}{\Omega^{k-1}(M)}$. 
We have an embedding of $\tilde{\Omega}^{k-2}(M)$ into $P_{k-3}$:
\begin{align}\label{s1}
0\rightarrow\tilde{\Omega}^{k-2}(M)\rightarrow P_{k-3}\rightarrow \tilde{\Omega}^{k-3}(M) \rightarrow 0
\end{align}
where $\tilde{\Omega}^{k-3}(M)=\faktor{P_{k-3}}{\tilde{\Omega}^{k-2}(M)}$,
because $\tilde{\Omega}^{k-2}(M)$ and $P_{k-3}$ shares the same socle: $soc(\tilde{\Omega}^{k-2}(M))\cong soc(\Omega^{k-2}(M))$ due to corollary \ref{topsoc}.
Notice that the short exact sequence \ref{s1} satisfies conditions in proposition \ref{temelteknik}, therefore there exists $P'_{k-4}\in\cP^{per}$ such that the sequence:
\begin{align}
0\rightarrow \tilde{\Omega}^{k-2}(M)\rightarrow P_{k-3}\rightarrow P'_{k-4}\rightarrow\tilde{\Omega}^{k-4}(M)\rightarrow 0
\end{align}
is exact by remark \ref{teknik1} where $P'_{k-4}\in\cP^{per}$ and $\tilde{\Omega}^{k-4}(M)$ denotes the quotient $\faktor{P'_{k-4}}{\tilde{\Omega}^{k-3}(M)}$.
We have an embedding of $\tilde{\Omega}^{k-4}(M)$ into $P_{k-5}$:
\begin{align}
0\rightarrow\tilde{\Omega}^{k-4}(M)\rightarrow P_{k-5}\rightarrow \tilde{\Omega}^{k-5}(M) \rightarrow 0
\end{align}
where $\tilde{\Omega}^{k-5}(M)=\faktor{P_{k-5}}{\tilde{\Omega}^{k-4}(M)}$
because $\tilde{\Omega}^{k-4}(M)$ and $P_{k-5}$ shares the same socle: $soc(\tilde{\Omega}^{k-4}(M))\cong soc(\Omega^{k-4}(M))$ due to corollary \ref{topsoc}.\\\\
In general, at the $i$th step we can apply proposition \ref{temelteknik} to:
\begin{align}
0\rightarrow\tilde{\Omega}^{k-2i+2}(M)\rightarrow P_{k-2i+1}\rightarrow \tilde{\Omega}^{k-2i+1}(M) \rightarrow 0
\end{align}
where ${\tilde{\Omega}^{k-2i+1}(M)}=\faktor{P_{k-2i+1}}{\tilde{\Omega}^{k-2i+2}(M)}$, to get the exact sequences:
\begin{align}
0\rightarrow \tilde{\Omega}^{k-2i+2}(M)\rightarrow P_{k-2i+1}\rightarrow P'_{k-2i}\rightarrow\tilde{\Omega}^{k-2i}(M)\rightarrow 0
\end{align} 
and then
\begin{align}
0\rightarrow\tilde{\Omega}^{k-2i}(M)\rightarrow P_{k-2i-1}\rightarrow \tilde{\Omega}^{k-2i-1}(M) \rightarrow 0
\end{align}
By the recursion, the last steps are :
\begin{align}
0\rightarrow \tilde{\Omega}^{1}(M)\rightarrow P_{0}\rightarrow\tilde{\Omega}^{0}(M)\rightarrow 0
\end{align} 
this can be completed to sequence by \ref{teknik1}:
\begin{align}
0\rightarrow \tilde{\Omega}^{1}(M)\rightarrow P_{0}\rightarrow P'\rightarrow\tilde{M}\rightarrow 0
\end{align}
Now,$\rho(\tilde{M})=1+\rho(\tilde{\Omega}^0(M))=k+1$. Since $k$ is odd, we obtained an indecomposable module $\tilde{M}$ which reaches a periodic syzygy in even number of steps i.e. $\Omega^{k+1}(\tilde{M})\in\bf\Omega^{per}$ and $\Omega^{k}(\tilde{M})\notin\bf\Omega^{per}$.
\end{proof}

\begin{corollary} Let $Z$ be an indecomposable $\Lambda$-module with $\rho(Z)=k$ where $k$ is an odd integer. Then there exists an indecomposable module $\tilde Z$ with $\rho(\tilde Z)=k+1$.
\end{corollary}

\subsection{Modification of projective resolutions when $\rho(X)$ is even}
To prove the theorem \ref{even} we do not need the modifications of the resolutions when $\rho(X)$ is even, anyhow we explain what we get if we apply the main techniques \ref{teknik1}, \ref{temelteknik} to the module $M$ in the following remark:
\begin{remark}\label{modificationofeven} The method we used to prove proposition \ref{modificationofodd} works for $k$ odd does not work for $k$ even. In other words, if an indecomposable module $M$ reaches to ${\bf\Omega}^{per}$ at an even number of steps i.e. $\rho(M)=2q$, then we can find a module $\tilde{M}$ such that $\rho(\tilde{M})$ is $2q,2q+2,\ldots$ i.e. even again. This is simply because of the remark \ref{teknik1}, the sequence of the form \ref{4stepres} implies the equality $\rho(\tilde{P}/Y)+2=\rho(X)$ where $X\cong\Omega^{2i}(M)$. Therefore, the even number of steps is preserved.
\end{remark}

\subsection{Modification of resolutions when projective dimension of $X$ is finite}
Now, we want to analyze the modules of the second kind, i.e. their resolutions stop at projective modules. Those have one of the below filtrations.

\begin{proposition} Let $X$ be a module of finite projective dimension. Let $P$ be projective module isomorphic to $\Omega^{\pdim X}(X)$. Possible (non-unique) filtrations of $P$ with respect to ${\bf\Omega}^{per}$ and $\bf\Delta$-modules are:

\begin{gather}\label{projtypes}
\Large{\begin{vmatrix}
   M  \\
    \Sigma
\end{vmatrix}}, 
\quad\Large{\begin{vmatrix}
   \Sigma_1  \\
    \Sigma_2
\end{vmatrix}} ,
\quad\Large{\begin{vmatrix}
   M_1 \\
   \Sigma \\
    M_2
\end{vmatrix}}
\end{gather}

where modules represented by $\Sigma$ are in ${\bf\Omega}^{per}$.
\end{proposition}

Precisely we have:
\begin{enumerate}[label=\arabic*)]
\item The first type projective modules $P$ contain a module $\Sigma$ of ${\bf\Omega}^{per}$ as submodule and $\faktor{P}{\Sigma}$ is not in ${\bf\Omega}^{per}$.
\item The second type projective modules are nontrivial extensions of two modules $\Sigma_1$, $\Sigma_2$ of ${\bf\Omega}^{per}$ such that $\Omega(\Sigma_1)\cong\Sigma_2$.
\item The third type projective modules have the following structure: only proper subquotients can be an element of ${\bf\Omega}^{per}$. 
\end{enumerate}
Notice that if a module $\Sigma\in{\bf\Omega}^{per}$ is a quotient of a projective, then $\Omega(\Sigma)\in{\bf\Omega}^{per}$. Hence, the following type of projective is impossible: $\begin{vmatrix}
    \Sigma \\
    M
\end{vmatrix}$. Also, in the above types \ref{projtypes}, $M$ and $M_1$ (respectively, $M_2$) can have submodules (respectively, quotients) from ${\bf\Omega}^{per}$. For example, by the first type we also keep track of modules of the form $\begin{vmatrix}
   M \\
   \Sigma_1 \\
   \Sigma_2
\end{vmatrix}$ etc.

\begin{proposition}\label{tip1}
Let $X$ be a module such that $\Omega^k\left(X\right)$ is projective (i.e. $\pdim X=k$) of the first type in $\ref{projtypes}$. If $k$ is odd, then there exists a module $\tilde{X}$ which has a projective resolution with the property that $\Omega^{k+1}(\tilde{X})\in{\bf\Omega}^{per}$ and $\Omega^{k}(\tilde{X})\notin{\bf\Omega}^{per}$.
\end{proposition}

\begin{proof}
We will show existence of module $\tilde{X}$ by using resolution of $X$. Since $\Omega^k(X)$ ($\cong PX_k$) has the shape of Type 1 in \ref{projtypes} the following sequence exists:
\begin{align}
0\rightarrow \Sigma \rightarrow PX_k \rightarrow \faktor{PX_k}{\Sigma} \rightarrow 0
\end{align}
$\Sigma$ is in ${\bf\Omega}^{per}$ hence it is a quotient of a projective of $\cP^{per}$. By the structure of type 1 projectives, $M$ is not in ${\bf\Omega}^{per}$. So we can apply the main techniques described in propositions \ref{temelteknik} and \ref{teknik1} to create module $\tilde{X}$ satisfying hypothesis in the proposition.
\end{proof} 

\begin{proposition}\label{tip2}
Let $X$ be a module such that $\Omega^k\left(X\right)$ is projective (i.e. $\pdim X=k$) of the second type in $\ref{projtypes}$. If $k$ is odd, then there exists a module $\tilde{X}$ which has a projective resolution with the property that $\Omega^{k+1}(\tilde{X})\in{\bf\Omega}^{per}$ and $\Omega^{k}(\tilde{X})\notin{\bf\Omega}^{per}$.
\end{proposition}

\begin{proof}
We will show existence of module $\tilde{X}$ by using resolution of $X$. Since $\Omega^k(X)$ ($\cong PX_k$) has the shape of Type 2 in \ref{projtypes} the following sequence exists:
\begin{align}
0\rightarrow \Sigma \rightarrow PX_{k-1} \rightarrow \faktor{PX_{k-1}}{\Sigma} \rightarrow 0
\end{align}
where $\Omega^k(X)\hookrightarrow  PX_{k-1}$.

$\Sigma$ is in ${\bf\Omega}^{per}$ hence it is a quotient of a projective of $\cP^{per}$. By the structure of $PX_{k-1}$, $\faktor{PX_{k-1}}{\Sigma}$ is not in ${\bf\Omega}^{per}$. So we can apply the main techniques described in propositions \ref{temelteknik} and \ref{teknik1} to create module $\tilde{X}$ satisfying hypothesis in the proposition.
\end{proof} 

\begin{proposition}\label{tip3}
Let $X$ be a module such that $\Omega^k\left(X\right)$ is projective (i.e. $\pdim X=k$) of the third type in $\ref{projtypes}$. If $k$ is odd,  then there exists a module $\tilde{X}$ which has a projective resolution with the property that $\Omega^{k+1}(\tilde{X})\notin{\bf\Omega}^{per}$.
\end{proposition}

\begin{proof}
We will show existence of module $\tilde{X}$ by using resolution of $X$. Since $\Omega^k(X)$ ($\cong PX_k$) has the shape of Type 3 in \ref{projtypes} the following sequence exists:
\begin{align}
0\rightarrow M_2 \rightarrow PX_{k}\cong \begin{vmatrix}
   M_1 \\
   \Sigma \\
    M_2
\end{vmatrix} \rightarrow \faktor{PX_{k}}{M_2}\cong \begin{vmatrix}
   M_1 \\
   \Sigma \end{vmatrix} \rightarrow 0
\end{align}

$M_2$ is a quotient of some projective in $\cP^{per}$ since it is proper subquotient of projective cover of $\Omega(\Sigma)$. By the structure of $PX_{k}$, $\faktor{PX_{k}}{M_2}$ is not in ${\bf\Omega}^{per}$, since $M_1\notin{\bf\Omega}^{per}$. So we can apply the main techniques described in propositions \ref{temelteknik} and \ref{teknik1} to create module $\tilde{X}$ such that $\Omega^{k+1}(\tilde{X})\cong M_2$ which is neither in ${\bf\Omega}^{per}$ nor projective module. Therefore $\rho(\tilde{X})>\pdim X$.
\end{proof} 

\subsection{Proof of the Theorem (A)}
Before proving the main theorem (A), we want to analyze one remaining case separately: $r=1$. So far, we studied all cyclic Nakayama algebras with REL, and number of relations were $\geq 2$.
\begin{proposition}\label{onerelation}
Let $\Lambda$ be a cyclic non self injective Nakayama algebra with $N$ nodes described by one monomial relation $\boldsymbol\alpha_{k_2}.\ldots\boldsymbol\alpha_{k_1}=0$. If the module $D\cong\begin{vmatrix}
    S_{k_{2}+1} \\
    \vdots  \\
    S_{k_{2}}
\end{vmatrix}$ of length $N$ (notice that it is a $\bf\Delta$-module) is a projective module, then $\gldim\Lambda=2$. Otherwise, it is of infinite projective dimension and $\varphi\dim\Lambda=2$.
\end{proposition}
\begin{proof}
Since there is one relation, this forces that there is one class \ref{classes} of projective modules, and socle is $S_{k_2}$ upto isomorphism. Let $M$ be arbitrary indecomposable nonprojective module. Consider the resolution:
\[\xymatrixcolsep{10pt}
\xymatrix{
P\Omega^2(M)\ar[rd] \ar[rr] && P\Omega(M)\ar[rd] \ar[rr] && PM \ar[rr] && M  \\
&\Omega^2(M)\ar[ru]& &\Omega(M) \ar[ru] & & 
}\]

All projectives have socle $S_{k_2}$, by corollary \ref{topsoc}, top and socle of $\Omega^2(M)$ are $S_{k_2+1}$ and $S_{k_2}$ respectively. If $D$ is projective, $D\cong P\Omega^2(M)$, hence global dimension is $2$. If $D$ is not projective, there exists a projective cover $P$ of $\Omega^2(M)$ such that it is some extension of $D$. The reason is the cyclic ordering \ref{c-ordering} of simples:
\begin{align*}
S_{k_1+1}\cdot\!\!\!\!>\cdots \cdot\!\!\!\!> S_{k_1} \cdot\!\!\!\!> S_{k_2+1}
\end{align*}
 The outer $S_{k_2+1}$ stands for top of $D$. Since there are $N$ modules between $S_{k_1+1}$ and $S_{k_1}$, this forces that there exists a projective module which has top $S_{k_2+1}$ and longer than $D$. Hence $D$ and its possible extensions are in ${\bf\Omega}^{per}$. This forces that $\varphi\dim\Lambda=2$.
\end{proof}

Now we prove the following:

\begin{theorem}\label{even}
If $\Lambda$ is a cyclic Nakayama algebra of infinite global dimension, then its $\varphi$-dimension is always an even number.
\end{theorem}
\begin{proof} By theorem \ref{onerelation}, we can assume that $r\geq 2$. \\
We analyze two cases separately: either $\varphi\dim\Lambda>\findim\Lambda$ or $\varphi\dim\Lambda=\findim\Lambda$. 
In the case of inequality it is enough to study resolutions of modules of infinite projective dimension.

We have shown that if an indecomposable module $M$ of infinite projective dimension reaches  ${\bf\Omega}^{per}$ at odd number steps $2q-1$, there exists $\tilde{M}$ such that $\tilde{M}$ reaches ${\bf\Omega}^{per}$ in $2q$ steps \ref{modificationofodd}. We conclude that $\varphi\dim\Lambda$ is always an even number if $\varphi\dim\Lambda >\findim\Lambda$.\\
We will show that the equality cannot  happen if $\findim\Lambda$ is odd. 
Assume to the contrary that there exists a module $X$ such that $\findim\Lambda=\pdim X$ and it is an odd number.
By the types of projectives in \ref{projtypes}, and propositions \ref{tip1}, \ref{tip2}, \ref{tip3}, it is guaranteed that $\varphi\dim\Lambda>\findim\Lambda$ by the construction of module $\tilde{X}$ satisfying $\rho(X)>\pdim X$. Hence they are not equal to each other. If $\pdim X$ is an even number, then by remark \ref{modificationofeven}, either $\varphi\dim\Lambda=\findim\Lambda$ or $\varphi\dim\Lambda=2+\findim\Lambda$, and in the both cases they are even numbers. 
\end{proof}

\subsection{Example}

We want to illustrate $\bf\Delta$-filtration and modification of resolutions in the following example. Let $\Lambda$ be cyclic Nakayama algebra over $N=8$ vertices and given by the relations:
\begin{align*}
\boldsymbol\alpha_3\boldsymbol\alpha_2\boldsymbol\alpha_1\boldsymbol\alpha_8\boldsymbol\alpha_7\boldsymbol\alpha_6\boldsymbol\alpha_5\boldsymbol\alpha_4\boldsymbol\alpha_3\boldsymbol\alpha_2\boldsymbol\alpha_1=0\\
\boldsymbol\alpha_6\boldsymbol\alpha_5\boldsymbol\alpha_4\boldsymbol\alpha_3\boldsymbol\alpha_2\boldsymbol\alpha_1\boldsymbol\alpha_8\boldsymbol\alpha_7\boldsymbol\alpha_6\boldsymbol\alpha_5\boldsymbol\alpha_4=0\\
\boldsymbol\alpha_8\boldsymbol\alpha_7\boldsymbol\alpha_6\boldsymbol\alpha_5\boldsymbol\alpha_4\boldsymbol\alpha_3\boldsymbol\alpha_2\boldsymbol\alpha_1\boldsymbol\alpha_8\boldsymbol\alpha_7\boldsymbol\alpha_6\boldsymbol\alpha_5=0\\
\boldsymbol\alpha_2\boldsymbol\alpha_1\boldsymbol\alpha_8\boldsymbol\alpha_7\boldsymbol\alpha_6\boldsymbol\alpha_5\boldsymbol\alpha_4\boldsymbol\alpha_3\boldsymbol\alpha_2\boldsymbol\alpha_1\boldsymbol\alpha_8\boldsymbol\alpha_7=0
\end{align*}

$\bf\Delta$-modules are: $\Delta_1\cong\begin{vmatrix}
   S_1 \\
   S_2  
  \end{vmatrix},\quad\Delta_2\cong\begin{vmatrix}
  S_3
  \end{vmatrix},\quad\Delta_3\cong\begin{vmatrix}
   S_4\\
   S_5\\
   S_6  
  \end{vmatrix},\quad\Delta_4\cong\begin{vmatrix}
   S_7\\
   S_8 
  \end{vmatrix}$.

In terms of $\bf\Delta$-filtration $\bf\Delta$-projectives are:
  \begin{align}
  P_1\cong\begin{vmatrix}
   \Delta_1 \\
   \Delta_2 \\
   \Delta_3\\
   \Delta_4\\
   \Delta_1\\
   \Delta_2
  \end{vmatrix},\qquad P_3\cong\begin{vmatrix}
   \Delta_2 \\
   \Delta_3 \\
   \Delta_4 \\
   \Delta_1\\
   \Delta_2\\
   \Delta_3 
  \end{vmatrix}, \qquad P_4\cong\begin{vmatrix}
  \Delta_3\\
  \Delta_4\\
  \Delta_1\\
  \Delta_2\\
  \Delta_3
  \end{vmatrix},\qquad P_7\cong\begin{vmatrix}
  \Delta_4\\
  \Delta_1\\
  \Delta_2\\
  \Delta_3\\
  \Delta_4\\
  \Delta_1
  \end{vmatrix}
  \end{align}
 By direct computation, $\varphi\dim\Lambda$ is $6$. But we want to illustrate modifications of resolutions. Recall notation $\Delta^j_i$ for the module with $\bf\Delta$-soc $\Delta_i$ and $\bf\Delta$-length $j$.

It is easy to check $\Omega(\Delta^1_2)\cong P_4$ and $\Delta^2_3\in{\bf\Omega}^{per}$. If we apply the proposition \ref{tip1}, we get the module $\Delta^5_4$  and $\rho(\Delta^5_4)=2>\pdim\Delta^1_2$. Observe that $\rho(\Delta^1_1)=3$. If we apply the proposition \ref{modificationofodd}, we get module $\Delta^5_4$, and $\rho(\Delta^5_4)=4$.


\section{Upper Bound For $\varphi$-dimension }
\subsection{Proof of the Theorem (B)}
Aim of this section is to prove:

\begin{theorem}\label{sharpbound}
Let $\Lambda$ be a cyclic Nakayama algebra given by $r$ relations and of infinite global dimension. Then a sharp bound for $\varphi\dim\Lambda$ is equal to $2r$.
\end{theorem}

To prove this, we will modify techniques which were used to show:
\begin{theorem}[W. Gustafson \cite{gust}]
Let $\Lambda$ be a cyclic Nakayama algebra over $N$ nodes of finite global dimension. Then global dimension of $\Lambda$ is smaller or equal than $2N-2$.
\end{theorem}
Here, we use the notation introduced at \cite{gust}.
Let $[j]$ denote the least positive residue of $j$ modulo $r$ whenever $j>0$. Otherwise $[r]=r$. Let $f$ be a function from $[1,r]$ to $[1,r]$ defined by $f(i)=[i+c_i]$, where $c_i$ is the $\bf\Delta$-length of $\bf\Delta$-projective module $P_i$. The fixed point set is nonempty: $Y=\left\{s\in [1,r]\,\vert f^t(s)=s\,\text{for some}\,t\in\Zz^{+}\right\}$ \label{set}. Now we give proof of theorem \ref{sharpbound}:
\begin{proof}
 The function $f$ is permutation on the set $\left\{f^d(1),\ldots,f^d(r)\right\}$ for some $d$ where $d\leq r-1$. Consider the projective resolution:
\begin{align*}
\cdots P_{f^2(j)}\rightarrow P_{f^2(i)}\rightarrow P_{f(j)}\rightarrow P_{f(i)}\rightarrow P_{j}\rightarrow P_{i}\rightarrow M
\end{align*}
We can assume that $M$ has $\bf\Delta$-filtration by proposition \ref{deltafilt}. Moreover it is enough to consider $\pdim M=\infty$, simply $\varphi\dim\Lambda\geq \findim\Lambda$ by the remark \ref{fidimfindim}.\\
The index of any projective module appearing in the resolution after $P_{f^d(i)}$ is periodic i.e. there exists $k$ such that if $f^{d}(i)=m$ then $f^{d+k}(i)=m$. Assume that, without loss of generality, $\Omega^{2d+2x}$ is a periodic syzygy i.e. $\Omega^{2d+2x}(M)\in\bf\Omega^{per}$. It is submodule of $P_{f^{d+x-1}(j)}$. Since $f$ permutes its index, $P_{f^{d+x-1}(j)}$ appears in the projective resolution again, therefore it is a periodic projective i.e. $P_{f^{d+x-1}(j)}\in\cP^{per}$. Hence $\Omega^{2d+2x-1}(M)$ is periodic syzygy again. These arguments work recursively, until $P_{f^d(i)}$. If $\Omega^{2d+1}(M)$ is periodic syzygy then $\Omega^{2d}(M)$ is also because $P_{f^d(i)}$ is in $\cP^{per}$. However, $f^{d-1}(j)$ does not have to be fixed by $f$, $P_{f^{d-1}(j)}$ does not have to appear in $\cP^{per}$.
We get $\rho(M)=2d$. Since $d$ can be at most $r-1$ by \cite{gust}, and $M\in\Omega^2$, the upper bound is $\varphi\dim\Lambda=\rho(M)+2=2r$.
\end{proof} 

Now we give an example inspired by \cite{gust} to show that it is a sharp bound. Consider the Nakayama algebra $\Lambda$ on $N$ vertices with Kupisch series $(2N+1,2N+1,\ldots,2N+1,2N)$. This has $r=N-1$, and $\varphi\dim\Lambda=\findim\Lambda=2r$, since the module of length $N+1$ with socle $S_N$ has projective dimension $2N-2$ and $\rho(S_N)=2N-2$.

We give another proof of theorem \ref{onerelation} as a corollary of theorem \ref{sharpbound}:
\begin{corollary} Let $\Lambda$ be a cyclic Nakayama algebra described by one monomial relation. Then $\varphi\dim\Lambda\leq 2$.
\end{corollary}
\begin{proof}
There is one relation, this forces that $\varphi\dim\Lambda\leq 2$ by theorem \ref{sharpbound}. 
\end{proof}

\begin{remark} In \cite{ringel} proposition 6, Ringel shows that assuming projectives are longer than the number of non-isomorphic simple modules, the Gorenstein dimension of a cyclic Nakayama algebra is $2d$ where $d$ is the smallest integer which makes the function $f$ periodic on the set \ref{set}. By the theorem \ref{even} and the result of \cite{ralf}, we obtain: If a cyclic Nakayama algebra of infinite global dimension is Gorenstein then its Gorenstein dimension is always an even number. 
\end{remark}


\end{document}